\documentclass[12pt]{article}
\usepackage[T1]{fontenc}
\usepackage{a4wide}
\usepackage{lipsum}

\title{New constructions of unbalanced $\{C_4,\theta_{3, t}\}$-free bipartite graphs}

\author{
Baran D\"uzg\"un \thanks{Hacettepe University, Graduate School of Science and Engineering, Beytepe 06800 Ankara, T\"urkiye. E-mail: {\tt baranduzgun@hacettepe.edu.tr}.}
\and
Ago-Erik Riet \thanks{Institute of Mathematics and Statistics, University of Tartu, 51009 Tartu, Estonia. E-mail: {\tt ago-erik.riet@ut.ee}. This work was supported by the Estonian Research Council grant PRG2531.}
\and
Vladislav Taranchuk \thanks{Department of Mathematics: Analysis, Logic and Discrete Mathematics, Ghent University, 9000 Ghent, Belgium. E-mail: {\tt vlad.taranchuk@ugent.be}.}
}

\makeatletter

\usepackage{amsfonts, amsmath, amssymb,latexsym,epsfig}
\usepackage{amsthm}
\usepackage{tikz}
\usetikzlibrary{graphs}
\usetikzlibrary{graphs.standard}
\usepackage{relsize}
\usepackage{verbatim}
\usepackage{enumerate}
\usepackage[skip=10pt plus1pt, indent=20pt]{parskip}

\usepackage{hyperref}

\newtheorem{thm}{Theorem}[section]
\newtheorem{cor}[thm]{Corollary}

\newtheorem{defn}[thm]{Definition}

\newcommand{\F}{\mathbb F}

\begin{document}

\maketitle
\thispagestyle{empty}


\begin{abstract}
    In 1979, Erd\H{o}s conjectured that if $m = O(n^{2/3})$, then $ex(n, m, \{C_4, C_6 \}) = O(n)$. This conjecture was disproven by several papers and the current best-known bounds for this problem are
    $$
    c_1n^{1 + \frac{1}{15}} \leq ex(n, n^{2/3}, \{C_4, C_6\}) \leq c_2n^{1 + 1/9}
    $$
    for some constants $c_1, c_2$. A consequence of our work here proves that 
    $$
    ex(n, n^{2/3}, \{ C_4, \theta_{3, 4} \}) = \Theta(n^{1 + 1/9}).
    $$
    More generally, for each integer $t \geq 2$, we establish that 
    $$
    ex(n, n^{\frac{t+2}{2t+1}}, \{ C_4, \theta_{3, t} \}) = \Theta(n^{1 + \frac{1}{2t+1}})
    $$
    by demonstrating that subsets of points $S \subseteq \text{PG}(n,q)$ for which no $t+1$ points lie on a line give rise to $\{ C_4, \theta_{3, t} \}$-free graphs, where PG$(n,q)$ is the projective space of dimension $n$ over the finite field of $q$ elements.

    
     
\end{abstract}

\section{Introduction}

Let $m, n$ be positive integers and $\mathcal{F}$ be a family of graphs. The bipartite Tur\'{a}n number $ex(m, n, \mathcal{F})$, is the maximum number of edges in a bipartite graph whose part sizes are $m$ and $n$ and such that it contains no graph in $\mathcal{F}$ as a subgraph. The function $ex(m, n, \mathcal{F})$ has been studied extensively for many different sets $\mathcal{F}$, but many questions remain. See the well-known survey by F\"{u}redi and Simonovits \cite{Furedi} for a history of the work done on these types of problems.

One of the most notorious cases of determining $ex(n, m, \mathcal{F})$ is when $\mathcal{F} = C_{2k}$ for some positive integer $k$. When $m = n$, the order of magnitude of $ex(n, n, \{C_{2k} \})$ is only known for $k = 2, 3, 5$ \cite{Furedi}, which coincides with the existence of special finite geometries called generalized polygons. When $m = n^a$ for any $a < 1$, and we allow $n \rightarrow \infty$, even less is known. The best constructions in this aforementioned case which yield $C_4$-free graphs come from the point-block incidence graphs of 2-designs.

In 1979, Erd\H{o}s conjectured that when $m = O(n^{2/3})$, then 
$ex(m, n, \{C_4, C_6 \}) = O(n)$ \cite{Erdos}. This was disproven first by de Caen and Sz\'{e}kely \cite{dCS} who constructed an infinite family of graphs which yielded
$$
cn^{1 + \frac{1}{57} + o(1)} \leq ex(n, n^{2/3}, \{C_4, C_6\})
$$
for a constant $c$.
Later this lower bound was improved by Lazebnik, Ustimenko and Woldar \cite{LUW} who constructed an infinite family of graphs which yielded
$$
cn^{1 + \frac{1}{15}} \leq ex(n, n^{2/3}, \{C_4, C_6\}).
$$
for some constant $c$.
It is worth mentioning that the class of graphs which is constructed in \cite{LUW} is very much related to the $(q, q^2)$-generalized quadrangle. In fact, it seems to have escaped the graph theory community that the existence of such quadrangles also yields the bound obtained in \cite{LUW} by taking an induced subgraph. The bounds on $ex(n ,n^{2/3}, \{C_4, C_6 \})$ have not budged in 30 years. In this paper, we establish results which suggest that the upper bound for this problem may be closer to the truth. 

A common generalization of the cycle $C_{2k}$, is the theta graph $\theta_{t, k}$ which is the graph consisting of two vertices joined by $t$ internally vertex-disjoint $k$-edge paths. While the lower and upper bounds for $ex(n, n, \{C_{2k} \})$ do not have matching orders of magnitude for all $k \neq 2, 3, 5$ \cite{Furedi}, it is known that for each $k$, there exists a (relatively large) constant $t = t(k)$ such that 
$$
ex(n, n, \theta_{k, t}) = \Theta(n^{1 + \frac{1}{k}}).
$$
which is the same order of magnitude as the upper bound for $ex(n, n, C_{2k})$ \cite{Conlon}. In \cite{Conlon}, Conlon uses a random algebraic method to construct infinite families of graphs not containing $\theta_{k, t}$, where $t$ is fixed, but large relative to $k$. Therefore, explicit constructions yielding the same bounds and with a smaller $t$ are of great interest. 

There has also been some recent progress on determining $ex(n, m , \theta_{t, k})$. Jiang, Ma, and Yepremyan \cite{JLY} proved that there exists a constant $c = c(k, t)$ such that
$$
ex(m, n, \{ \theta_{k, t} \}) \leq \left\{ \begin{array}{cc}
    c[(mn)^{\frac{k+1}{2k} + m + n}] & \text{if }k\text{ is odd} \\
     c[(mn)^{\frac{k+2}{2k}n^{\frac{1}{2}} + m + n}] &  \text{if }k\text{ is even}
\end{array} \right.
$$
and when $k = 3$, they obtained that 
$$
ex(m, n, \{ \theta_{3, t} \}) \leq 144t^3((mn)^{\frac{2}{3}} + m + n).
$$
Theodorakopoulos \cite{Theo} extended the random algebraic methods used in \cite{Conlon} to prove that for each odd positive integer $k$ and rational number $a$ satisfying $\frac{k-1}{k+1} < a < 1$, there exists a constant $c = c(k)$ such that 
$$
ex(n, n^a, \{ \theta_{k, c_k}\}) = \Theta(\left(n^{1 + a}\right)^{\frac{k+1}{2k}} )
$$

Here we prove that subsets of points $S$ of the projective space $\text{PG}(n, q)$ satisfying the condition that no $t +1$ points of $S$ lie on a line, produce $\{C_4, \theta_{3, t} \}$-free graphs via their linear representations.  In particular, this implies the following theorem.

\begin{thm}
    Let $q$ be a prime power and $t, n$ be positive integers. Suppose that $S$ is a subset of points of $\emph{PG}(n, q)$ satisfying the condition that no $t+1$ points of $S$ lie on a common line. Then 
    $$
    |S|q^{n+1} \leq ex(q^{n+1}, |S|q^n, \{C_4, \theta_{3, t} \} ).
    $$
\end{thm}

We remark that such sets with many elements are known to exist as shown in Lin and Wolf \cite{LW}.  Consequently, we obtain our result.

\begin{thm}\label{MainThm}
    Let $t \geq 2$ be a positive integer. 
    Then 
    $$
    ex(n, n^{\frac{t+2}{2t+1}}, \{C_4, \theta_{3, t}  \}) = \Theta(n^{1 + \frac{1}{2t + 1}}).
    $$
\end{thm}

Finally, our theorem implies the following corollary which suggests that the true value of $ex(n, n^{2/3}, \{ C_4, C_6 \})$ may be closer to the best-known upper bound.

\begin{cor}
    We have 
    $$
    ex(n, n^{2/3}, \{ C_4, \theta_{3, 4} \}) = \Theta(n^{1 + \frac{1}{9}}).
    $$
\end{cor}

\section{Linear Representations of Point Sets}

Let $q$ be a prime power, let $\F_q^{n+1}$ denote the vector space of dimension $n+1$ over the finite field $\F_q$, and let PG$(n,q)$ be the corresponding projective space.

\begin{defn}
    Let $S$ be a set of points of $\emph{PG}(n, q)$ and embed $\emph{PG}(n, q)$ as a hyperplane into $\emph{PG}(n+1, q)$. A \emph{linear representation} of $S$ is the geometry whose points are all the points in $\emph{PG}(n+1, q)\setminus \emph{PG}(n, q)$ and the lines are all the lines of $\emph{PG}(n+1, q)$ which intersect $\emph{PG}(n, q)$ in precisely one point, namely a point of $S$. 
\end{defn}

\noindent \textbf{Remark}:  Observe that if two lines in PG$(n+1, q)$ not contained in PG$(n,q)$ intersect in a point of $S$, then they are parallel in the 
linear representation of $S$.

From this geometry we may build its point-line incidence graph. This graph is bipartite with bipartition classes given by the  points and the lines of the geometry. A point will be adjacent to a line in the graph if they are incident in the geometry, i.e.~the point is on the line. Denote this graph by $\Gamma_{S, n, q}$. It can easily be verified that $\Gamma_{S, n, q}$ will have the following properties:

\begin{enumerate}
    \item There are $q^{n+1}$ point vertices, each of degree $|S|$. 
    \item There are $|S|q^n$ line vertices, each of degree $q$.
\end{enumerate}

\begin{thm}
    Let $S$ be a subset of points of  $\emph{PG}(n, q)$ such that any line in $\emph{PG}(n, q)$ intersects $S$ in at most $t$ points. Then $\Gamma_{S, n, q}$ is $\{C_4, \theta_{3, t} \}$-free.
\end{thm}

\begin{proof}
    Note that by construction, the linear representation of $S$ is a geometry in which any two lines intersect in at most one point and any two points lie on at most one line. Thus, the graph $\Gamma_{S, n, q}$ is necessarily $C_4$-free.

    Suppose that $\Gamma_{S, n, q}$ contains a $\theta_{3, t}$. This implies there exists a point vertex $r$ and line vertex $\ell$ between which there are $t$ vertex-disjoint paths of length 3. Geometrically, this implies that the linear representation of $S$ contains a configuration consisting of a point  $r$, $t$ lines which contain $r$, call them $m_1, m_2, \dots, m_t$, all of which intersect the line $\ell$ in our geometry, and such that $r$ is not on $\ell$.
Note that the set of lines $\ell, m_1, m_2, \dots, m_t$ all pairwise intersect, and so no two can be parallel. Thus, in PG$(n + 1, q)$, each of the lines $\ell, m_1, \dots, m_t$ contains a distinct point in $S$.

Note that all of the lines of this configuration lie in a common plane, $\Pi$, the plane in PG$(n+1,q)$ spanned by $r$ and $\ell$. Since $\Pi$ is not contained in the hyperplane PG$(n, q)$, it intersects PG$(n, q)$ in a line, call it $\ell_{\infty}$. But this implies that there is a set of $t+1$ points in $S$ (one for each line $\ell, m_1, \dots, m_t$) which lie on $\ell_{\infty}$. Since we assumed any line in PG$(n, q)$ intersects $S$ in at most $t$ points, this is a contradiction. Thus $\Gamma_{S, n, q}$ is also $\theta_{3, t}$-free.
\end{proof}

In \cite{LW}, the authors obtained the following result. For completeness, we give an explicit construction of such a set for all integers $t \geq 2$ and prime powers $q > t$.

\begin{thm}
    Let $q$ be a prime power and $t > q$ be a positive integer. Then there exists a subset $S$ of points of \emph{PG}$(t + 1, q)$ of size $q^t$ such that no $t+1$ points of $S$ lie on a line. 
\end{thm}

It is well-known that the field $\F_{q^t}$ can be viewed as a vector space over $\F_q$. Fix any basis, $\alpha_1, \dots, \alpha_t$ for $\F_{q^t}$ over $\F_q$. For each $x$ in $\F_{q^t}$, denote by $x|_q$ the vector of the field reduced elements of $x$. That is, $x|_q = (x_1\alpha_1 + \cdots + x_t\alpha_t)|_q = (x_1, \dots, x_t)$. Denote by $N$ the norm function from $\F_{q^t}$ to $\F_q$, i.e.~$N(x)=x^{(q^t-1)/(q-1)}=x\cdot x^q\cdot x^{q^2}\cdot\ldots \cdot x^{q^{t-1}}$.

\begin{thm}
    Let $t \geq 2$ be a positive integer and $q > t$ be a prime power. Then the set 
    $$
    S = \{ (1, x|_q, N(x)) : x \in \F_{q^t} \}
    $$
    as a subset of points of \emph{PG}$(t+1, q)$ contains no $t + 1$ points on a line.
\end{thm}

\begin{proof}
    We will omit the notation of field reduction to avoid getting bogged down in notation. Take any two vectors in $S$, call them $(1, x, N(x))$ and $(1, y, N(y))$. These two points lie on some line in PG$(t+1, q)$. We will count how many other points $(1, z, N(z))$ in $S$ can lie on this same line. If $(1, z, N(z))$ lies on the same line, it implies that the matrix
    $$
    \begin{bmatrix}
        1 & 1 & 1 \\
        x & y & z \\
        N(x) & N(y) & N(z)
    \end{bmatrix}
    $$
    has an element in its null space $v = [a, b, c]$ with each $a, b, c \in \F_q^*$.  Without loss of generality, we assume $c = -1$. Thus we obtain a linear system 
    \begin{align*}
        a + b  &= 1, \\
        ax + by  &= z, \\
        aN(x) + bN(y) &= N(z).
    \end{align*}
    This system implies that
    \begin{equation}\label{MainEq}
    aN(x) + (1 - a)N(y) = N(ax + (1 - a)y)   
    \end{equation}
    Since $x$ and $y$ are fixed, that leaves $a$ as the only variable. Since $a \in \F_q$, then $a^q = a$ and so expanding (\ref{MainEq}), we obtain a polynomial equation in $a$ of degree $t$. In particular, the largest degree term $a^t$ has coefficient $N(x-y) \neq 0$. Observe that $a = 0$ and $a = 1$ are both solutions to (\ref{MainEq}), which are not valid choices of $a$ for us (since $a, b \neq 0$). Consequently, there are at most $t - 2$ valid choices of $a$ which solve (\ref{MainEq}). Each $a$ determines a unique $z = ax + (1-a)y$, so any line contains at most $t$ points of $S$.
\end{proof}

Consequently, the linear representation of this set produces a biregular, unbalanced bipartite graph with bipartition class sizes $m = q^{t+2}$, $n = q^{2t + 1}$, and $q^{2t + 2}$ edges which is $\{C_4, \theta_{3, t}\}$-free. By an application of Bertrand's postulate, we obtain that for each $t$, there exists a constant $c_t$ such that 
$$
c_tn^{1 + \frac{1}{2t+1}} \leq ex(n, n^{\frac{t+2}{2t + 1}}, \{C_4, \theta_{3, t} \}).
$$
For each such $t$, the order of magnitude matches the upper bounds given in \cite{JLY}. Thus we prove Theorem \ref{MainThm} 
$$
ex(n, n^{\frac{t+2}{2t + 1}}, \{C_4, \theta_{3, t} \}) = \Theta(n^{1 + \frac{1}{2t+1}}).
$$

\end{document}